\documentclass[reqno]{amsart}
\usepackage{amsmath, amsthm, amssymb, amstext}

\usepackage{hyperref,xcolor}% http://ctan.org/pkg/{hyperref,xcolor}
\hypersetup{
	pdfborder={0 0 0},
	colorlinks,
	linkcolor=red,
	citecolor=blue
}
\usepackage{enumitem}
\usepackage{graphicx}
\usepackage{circuitikz}
\newtheorem{theorem}{Theorem}
\newtheorem{rem}[theorem]{Remark}
\newtheorem{lemma}[theorem]{Lemma}
\newtheorem{proposition}[theorem]{Proposition}

\newcommand{\N}{\mathbb{N}}

\newcommand{\R}{\mathbb{R}}

\numberwithin{theorem}{section}
\numberwithin{equation}{section}

\title[Integro-differential sweeping process]{Lipschitz multivalued perturbations of integro-differential prox-regular sweeping processes}

\author[T. Haddad]{Tahar Haddad}
\address[T. Haddad]
{Laboratoire LMPEA, Facult\'e des Sciences Exactes et Informatique, Universit\'e Mohammed Seddik Benyahia,
Jijel, B.P. 98, Jijel 18000, Alg\'erie.}
\email[Corresponding author]{haddadtr2000@yahoo.fr}

\author[S. Gaouir]{Sarra Gaouir}
\address[S. Gaouir]
{Laboratoire LMPEA, Facult\'e des Sciences Exactes et Informatique, Universit\'e Mohammed Seddik Benyahia,
Jijel, B.P. 98, Jijel 18000, Alg\'erie.}
\email{saragaouir1997@gmail.com}

\author[L. Thibault]{Lionel Thibault}
\address[L. Thibault]
{Universit\'{e} de Montpellier, Institut Montpelli\'{e}rain Alexander Grothendieck 34095 Montpellier CEDEX 5 France.}
\email{lionel.thibault@umontpellier.fr}

\begin{document}
	
	\begin{abstract}
	  Integro-differential sweeping processes with prox-regular sets in Hilbert spaces have been the subject of various recent studies. Diverse applications of such differential inclusions to complementarity problems, electrical circuits, frictionless contact, can be found in the literature. Here we provide a general theorem of existence of solution for such processes perturbed by a Lipschitz multimapping with nonconvex values.
	\end{abstract}
	
	\keywords{Sweeping process, proximal normal, prox-regular set, Lipschitz mapping, Lipschitz multimapping, Volterra integral}
	
	\maketitle
	
	%***********************************************************************************************************************************
	\section{Introduction}
	  The Moreau's sweeping process differential inclusion
\begin{equation}\label{eq-Sweep1}
   -\frac{dx}{dt}(t) \in N_{C(t)}x(t), \quad\text{with}\; x(0)=x_0 \in C(T_0)	
\end{equation}
modeling in mechanics a point $x(t)$ swept by a moving set $C(t)$ (depending on time $t\in I:=[T_0,T]$) of a Hilbert space $H$ has been introduced in \cite{More}. Its well-posedness, i.e. the existence and uniqueness of solution, has been proved in the same paper under the convexity of the set $C(t)$ and under either its absolutely continuous variation (ACV, for short) or its right continuous bounded variation (RCBV, for short). When the sets $C(t)$ are prox-regular (resp. convex), it is generally said that \eqref{eq-Sweep1} is the formulation of the {\em prox-regular (resp. convex) sweeping process}. Under the absolutely continuous variation assumption, the well-posedness in the Hilbert setting of the prox-regular sweeping process has been shown in \cite{Colo-Gonc}. In finite dimensions, perturbations of the ACV prox-regular sweeping process with scalarly upper semicontinuous multimappings with compact convex values were studied in \cite{Thib-JDE}. Under the inclusion of such a  multivalued perturbation into a fixed compact set, the ACV study was extended to Hilbert spaces in \cite{Boun-Thib}; see also \cite{Brog-Tanw,Jour-Vilc,Nacr}. The paper \cite{Edmo-Thib-JDE} developed in Hilbert spaces
 the RCBV case with multivalued perturbations satisfying suitable compact growth conditions; see also \cite{Nacr}. The well-posedness of perturbed ACV prox-regular sweeping processes in Hilbert spaces has been established in \cite{Edmo-Thib-MathProg} for Lipschitz single-valued perturbations and this well-posedness has been extended in \cite{Adly-Nacr-Thib,Nacr-Thib} under the RCBV property of the moving set $C(t)$.

   Recently, the integro-differential prox-regular sweeping process
\begin{equation}\label{eq-Sweep2}
   \begin{cases}
		-\frac{dx}{dt}(t) \in N_{C(t)}x(t) +f_1(t,x(t)) + \int_{T_0}^tf_2(t,s,x(s))\,ds \\
		x(T_0)=x_0 \in C(T_0)
	 \end{cases}
\end{equation}
has been studied in \cite{Bouach_Haddad_Thibault,Boua-Hadd-Thib-JOTA,Boua-Hadd-Thib-COT} with prox-regular sets $C(t)$ of a Hilbert space $H$,  Lipschitz mappings $f_1(t,\cdot)$ and $f_2(t,s,\cdot)$ with values in $H$. Above $\int_{T_0}^tf_2(t,s,x(s))\,ds$ appears as a general integral of Volterra-type. Previously, the situation of the other less general integral $\int_{T_0}^tf_{0,2}(s,x(s))\,ds$ (with $f_1\equiv 0$) was analyzed in
\cite{Bren,Colo-Koza}.  Our objective in the present paper is to consider the situation when the integro-differential prox-regular sweeping process \eqref{eq-Sweep2} is perturbed by a Lipschitz multimapping
$F(t,\cdot)$ with nonconvex values in $H$, say the differential inclusion
\begin{equation}\label{eq-Sweep3}
   \begin{cases}
		-\frac{dx}{dt}(t) \in N_{C(t)} x(t) +f_1(t,x(t)) + \int_{T_0}^tf_2(t,s,x(s))\,ds +F(t,x(t)) \\
		x(T_0)=x_0 \in C(T_0).
	 \end{cases}
\end{equation}
The presence of both the mapping $f_1$ and the multimapping $F$ will be justified in Remark \ref{rem-f1-F}.
After recalling in Section \ref{sect-prel} diverse features on prox-regular sets and measurable multimappings and establishing a needed Gronwall-type inequality,
the proof of existence of an absolutely continuous solution of \eqref{eq-Sweep3} is developed in
Section \ref{sect-main}. In \cite{Mans-Keci-Hadd}  the perturbation of the ACV prox-regular integro-differential sweeping process \eqref{eq-Sweep2} with a maximal monotone operator $A$ was considered instead of the nonconvex valued Lipschitz multimapping $F(t,\cdot)$. The paper \cite{Tolst2017} studied  with such a perturbed multimapping $F(t,\cdot)$ the ACV evolution inclusion
\begin{equation*}
   \begin{cases}
	  -\frac{dx}{dt}(t) \in \partial \varphi_t(x(t))+ F(t,x(t)) \\
		x(T_0)=x_0\in \mathrm{dom}\,\varphi_t,
	 \end{cases}
\end{equation*}
where  $\partial \varphi_t$ is the subdifferential of a time dependent convex function
$\varphi_t: H \to \mathbb{R}\cup\{+\infty\}$ and
$\mathrm{dom}\,\varphi_t:=\{u\in H: \varphi_t(u) <+\infty\}$. The similar ACV problem with a time dependent maximal monotone operator $A(t)$ in place of $\partial \varphi_t$ is examined in \cite{Cast-Said}. Our approach here utilizes some basic ideas in the papers \cite{Fili,Ioff,Tolst2017}.

		\section{Preliminaries}\label{sect-prel}
	
	Let $H$ be a (real) Hilbert space and let $\langle \cdot,\cdot\rangle$ be its inner product and $\|\cdot\|$ its associated norm. The closed ball of $H$ with center $x\in H$ and radius $\delta >0$ is denoted $B[x,\delta]$ while $\mathbb{B}$ stands for the closed unit ball $B[0,1]$ of $H$. The distance function $d_S$ to a subset $S$ of $H$ is given by $d_S(x):=\inf_{y\in S}\|x-y\|$.
For any cone $P$ of $H$ (i.e., $tP \subset P$ for any real $t>0$), the polar of $P$ is defined as
$$
   P^{\circ}:=\{\varsigma \in H: \langle \varsigma, u \rangle \leq 0,\;\forall u\in P\}.
$$
	Let $S$ be a nonempty closed subset of $H$. For $x\in S$, a vector
$\varsigma \in H$ is a {\em proximal normal} to $S$ at $x$ provided that there exists a real $\eta >0$ such that $x\in \mathrm{Proj}_S(x+\eta\varsigma)$, where $\mathrm{Proj}_S(y)$ is the set of all nearest points from $S$ to $y$. The cone $N_S(x)$ of all proximal normals to $S$ at $x$ is called the {\em proximal normal cone} to $S$ at $x$ (see, e.g., \cite{Thib-book1}). For $x\not\in S$, one puts $N_S(x)=\emptyset$.

  When a certain uniformity of $\eta>0$ holds, in the sense that there exists an extended real $R>0$ such that for any $x\in S$ and any $\varsigma\in N_S(x)$ with $\|\varsigma\|\leq 1$
$$
     x\in \mathrm{Proj}_S(x+t\varsigma)		\quad\text{for all}\; t\in ]0,R[,
$$		
one says that the set $S$ is $R$-prox-regular (see, e.g., \cite{Poli-Rock-Thib,Colo-Thib,Thib-book2}). For prox-regular sets all basic normal cones in variational analysis coincide. The closed set $S$ is convex if and only if it is $\infty$-prox-regular. Sets which are $R$-prox-regular are known to enjoy a very long list of fundamental variational regularity properties, for which we refer the reader to,
e.g., \cite{Poli-Rock-Thib,Colo-Thib,Thib-book2}.
	
	One of those variational regularity properties needed in the paper is the so-called hypomonotonicity of the normal cone (see, e.g., \cite{Poli-Rock-Thib,Colo-Thib,Thib-book2}) as stated in the following proposition.

\begin{proposition}\label{prop-Hypo}
Given $R>0$ a nonempty closed set $ S $ in $ H $ is $R$-prox-regular if and only if  for any $ x_{i}\in S $, $ \varsigma_{i}\in N_{S}x_{i} $  with $ i=1,2 $ one has
the hypomonotonicity property
\begin{equation*}
\langle \varsigma_{2}-\varsigma_{1} , x_{2}-x_{1} \rangle\geq -\dfrac{1}{2}\bigg(\dfrac{\lVert \varsigma_{2} \rVert + \lVert \varsigma_{1} \rVert }{R}\bigg)\lVert x_{2}-x_{1} \rVert^{2} .
\end{equation*}
\end{proposition}

 Another important property of a prox-regular set $S$ in $H$ is that its proximal normal cone $N_Sx$ at $x\in S$ coincides with its Clarke normal cone $\mathcal{N}_Sx$ as said in part above, see, e.g., \cite{Thib-book1,Thib-book2} for the definition and the property; in particular, $N_Sx$ is closed and convex.
	Further, for any point $x$ in the prox-regular set $S$ it is known that the directional derivative of its distance function $d_S$ at $x$
\begin{equation}\label{eq-DeriDist}
    d'_S(x;u):=\lim_{\tau \downarrow 0}\tau^{-1}d_S(x+\tau{u})
\end{equation}
exists for every $u\in H$ (see, e.g., \cite[Theorem 8.98]{Thib-book2}), so the Bouligand-Peano tangent cone
$T_Sx$ to $S$ at $x$ (see, e.g., \cite{Thib-book1} for the definition) can be written as
$$
	 T_Sx=\{u \in H: d'_S(x;u)=0\},
$$
and hence the coincidence of the proximal normal cone $N_Sx$ with the corresponding Clarke normal cone
along with the inclusions $N_Sx\subset (T_Sx)^{\circ} \subset \mathcal{N}_Sx$ entail that
\begin{equation}\label{eq-NormTang}
      N_Sx=(T_Sx)^{\circ}=\{u\in H: d'_S(x;u)=0\}^{\circ}.
\end{equation}

		Now let us recall the following lemma which is a consequence of Gronwall's lemma.
		
	\begin{lemma}\label{lem-Grown1}
		Let  $I$ be an interval of $\mathbb{R}$ and let $T_0 \in I$. Let $ \rho : I\rightarrow \mathbb{R} $  be a nonnegative locally absolutely continuous function and let $ b_{1}, b_{2}, a : I\rightarrow \mathbb{R}_{+} $ be non-negative locally Lebesgue integrable functions. Assume that
	\begin{equation}\label{23}
	\dot{\rho}(t)\leq a(t)+b_{1}(t)\rho(t)+b_{2}(t)\int_{T_{0}}^{t}\rho(s)\,ds,\hspace{0.3cm} a.e.\,\, t\in I.
	\end{equation}	
		Then for all $ t\in I$, one has
		\begin{equation*}
		\begin{aligned}
		\rho(t)\leq  \rho(T_{0})\,\exp\bigg(\int_{T_{0}}^{t}(b(\tau)+1)\,d\tau\bigg) + \int_{T_{0}}^{t}a(s)\,\exp\bigg(\int_{s}^{t}(b(\tau)+1)\,d\tau\bigg)\,ds,
		\end{aligned}
		\end{equation*}	
		where $ b(t):=\max\{b_{1}(t),b_{2}(t)\} $ for $ t\in I$.
	\end{lemma}
	
   We will also need the following Gronwall-like inequality.
	
	\begin{lemma}[Gronwall-like differential inequality]\label{lem-Grown2}
		Let $I$  be an interval of $\mathbb{R}$, let $T_0\in I$ and let $ \rho : I\rightarrow \mathbb{R} $  be a non-negative locally absolutely continuous function.  Let $ K_{1},K_{2}, K_{3} : I\rightarrow \mathbb{R}_{+} $ be non-negative locally Lebesgue integrable functions such that
		\begin{equation}\label{5}
		\dot{\rho}(t)\leq  K_{1}(t)\rho(t)+K_{2}(t)\sqrt{\rho(t)}+K_{3}(t)\sqrt{\rho(t)}\int_{T_{0}}^{t}\sqrt{\rho(s)}\,ds, \hspace{0.3cm} a.e.\,\, t\in I.
		\end{equation}	
		Then for all $ t\in I$, one has with
		$ K(t):=\max\bigg\{\dfrac{K_{1}(t)}{2},\dfrac{K_{3}(t)}{2}\bigg\} $
		\begin{equation*}
		\begin{aligned}
		\sqrt{\rho(t)}&\leq  \sqrt{\rho(T_{0})}\,\exp\bigg(\int_{T_{0}}^{t}(K(s)+1)\,ds\bigg) \!+\! \int_{T_{0}}^{t}\dfrac{K_{2}(s)}{2}\exp\bigg(\int_{s}^{t}(K(\tau)+1)\,d\tau\bigg) ds.
		\end{aligned}
		\end{equation*}
	\end{lemma}
	\begin{proof}
	    Considering the locally absolutely continuous function $\rho_{\epsilon}(\cdot):=\rho(\cdot) +\varepsilon$, we see that the assumption \eqref{5} is fulfilled with $\rho_{\varepsilon}(\cdot)$ in place of $\rho(\cdot)$. We may then suppose that $\rho(\cdot) >0$ on $I$. Putting $\sigma(t):=\sqrt{\rho(t)}$ for all $t\in I$, the function $\sigma$ is locally absolutely continuous on $I$, and  for a.e. $t\in I$ we have
$$ \dot{\sigma}(t)=\dfrac{\dot{\rho}(t)}{2\sqrt{\rho(t)}}= \dfrac{\dot{\rho}(t)}{2\sigma(t)},
 \quad\text{so}\; 2\sigma(t)\dot{\sigma}(t)= \dot{\rho}(t).
$$
Further,  $ \sigma(t)^{2}=\rho(t)$ for all $t\in I$. Then, from \eqref{5} we have for a.e. $ t\in I $
			\begin{equation*}
			2\sigma(t)\dot{\sigma}(t)\leq K_{1}(t)\sigma(t)^{2}+K_{2}(t)\sigma(t)
			 +K_{3}(t)\sigma(t)\int_{T_{0}}^{t}\sigma(s)\,ds.
			\end{equation*}
			Therefore, for a.e. $t\in I$ we have
			\begin{equation}\label{eq-1star}
			\dot{\sigma}(t)  \leq \dfrac{K_{2}(t)}{2}+\dfrac{K_{1}(t)}{2}\sigma(t)
			+\dfrac{K_{3}(t)}{2}\int_{T_{0}}^{t}\sigma(s)\,ds.
			\end{equation}	
			Letting  $ K(t):=\max\bigg\{\dfrac{K_{1}(t)}{2},\dfrac{K_{3}(t)}{2}\bigg\} $,  $ a(t) :=\dfrac{K_{2}(t)}{2}$,  Lemma \ref{lem-Grown1} furnishes 	 the desired inequality of the lemma.
	\end{proof}
		
		   Given a measurable space $(\mathfrak{T}, \mathcal{T})$, a multimapping
	$\Gamma: \mathfrak{T}\rightrightarrows Y$ from $\mathfrak{T}$ into a complete separable
	metric space $Y$ is called $\mathcal{T}$-measurable if
	$\Gamma^{-1}(U)\in \mathcal{T}$  for every open set $U$ of $Y$, where
	$\Gamma^{-1}(U):=\{x\in H: \Gamma(x)\cap U \neq \emptyset\}$. If $\Gamma$ takes nonempty values, a mapping
	$\sigma: \mathfrak{T}\to Y$ is a selection of $\Gamma$ provided that $\sigma(t)\in \Gamma(t)$ for all
	$t \in \mathfrak{T}$. The set of all $\mathcal{T}$-measurable selections of $\Gamma$ will be denoted
	$\mathrm{Sel}_{\mathcal{T}}\Gamma$. If $\mathcal{T}$ is the Lebesgue $\sigma$-field $\mathcal{L}(\mathfrak{T})$ of a Borel subset $\mathfrak{T}$ of $\mathbb{R}^N$, we will denote $\mathrm{Sel}\,\Gamma$ for short.
	
	     We state in the following theorem some fundamental properties of measurable multimappings, for which we refer to Theorems III.9  and III.30 in \cite{Cast-Vala}. Denoting $\mathcal{B}(Y)$ the Borel $\sigma$-field of $Y$, the assertions (c) and (d) are related to the
$\mathcal{T}\otimes\mathcal{B}(Y)$-measurability of	the graph of $\Gamma$
$$
   \mathrm{gph}\,\Gamma:=\{(t,y)\in\mathfrak{T}\times{Y}: y\in \Gamma(t)\}.
$$

\begin{theorem}\label{theor-measurable}
Let $(\mathfrak{T}, \mathcal{T})$ be a measurable space, let $Y$ be a complete separable metric space, and let $\Gamma: \mathfrak{T}\rightrightarrows Y$ be a multimapping.\\
(a) If $\Gamma$ is closed valued, then $\Gamma$ is $\mathcal{T}$-measurable if and only if for each $y\in Y$ the function $t\mapsto d(y,\Gamma(t))$ is $\mathcal{T}$-measurable.\\
(b) If the values of $\Gamma$ are nonempty and closed in $Y$, then $\Gamma$ is $\mathcal{T}$-mesurable if and only if its admits a Castaing $\mathcal{T}$-representation, i.e. a sequence $(\sigma_n)_{n\in \mathbb{N}}$
 of $\mathcal{T}$-measurable selections of $\Gamma$ such that
$$
    \Gamma(t)=\mathrm{cl}_Y\{\sigma_n(t):n\in \mathbb{N}\}\quad \text{for every}\; t\in \mathfrak{T}.
$$
(c) If the values of $\Gamma$ are closed in $Y$ and if the $\sigma$-field $\mathcal{T}$ is complete with respect to a $\sigma$-finite positive measure, then
$\Gamma$ is $\mathcal{T}$-measurable if and only if its graph $\mathrm{gph}\,\Gamma$ belongs to $\mathcal{T}\otimes\mathcal{B}(Y)$.\\
(d) If the values of $\Gamma$ are nonempty (maybe non-closed), if the $\sigma$-field $\mathcal{T}$ is complete with respect to a $\sigma$-finite positive measure, and if its graph $\mathrm{gph}\,\Gamma$ belongs to $\mathcal{T}\otimes\mathcal{B}(Y)$, then $\Gamma$ admits at least one $\mathcal{T}$-measurable selection.\\
(e) If the $\sigma$-field $\mathcal{T}$ is complete with respect to a $\sigma$-finite positive measure and if $(\Gamma_j)_{j\in J}$ is a countable collection of $\mathcal{T}$-measurable multimappings from
$\mathfrak{T}$ into $Y$ with closed values, the multimapping $t\mapsto \bigcap_{j\in J}\Gamma_j(t)$ is
$\mathcal{T}$-measurable.
\end{theorem}

Let $\Gamma_1$ and $\Gamma_2$ be two multimappings from $\mathfrak{T}$ into $Y$. Assume that the $\sigma$-field  $\mathcal{T}$ is complete with respect to a $\sigma$-finite positive measure, that $\Gamma_1$ is $\mathcal{T}$-measurable and closed valued, and that the graph of the multimapping $\Gamma_2$ belongs to $\mathcal{T}\otimes\mathcal{B}(Y)$. Then for any $\mathcal{T}$-measurable selection $z(\cdot)$ of $\Gamma(\cdot):=\Gamma_1(\cdot)+\Gamma_2(\cdot)$, the multimapping $\Gamma_3(\cdot):=z(\cdot)-\Gamma_1(\cdot)$ is
$\mathcal{T}$-measurable (by Theorem \ref{theor-measurable}(b)), hence the graph of $t\mapsto \Gamma_2(t)\cap \Gamma_3(t)$ is
$\mathcal{T}\otimes\mathcal{B}(Y)$ measurable (by Theorem \ref{theor-measurable}(f)). Therefore, by Theorem \ref{theor-measurable}(d) there exists a $\mathcal{T}$-measurable selection
$z_2(\cdot)$ of $t\mapsto \Gamma_2(t)\cap\Gamma_3(t)$. The mapping $z_1(\cdot):=z(\cdot)-z_2(\cdot)$ is a
$\mathcal{T}$-measurable selection of $\Gamma_1$, and $z(\cdot)=z_1(\cdot)+z_2(\cdot)$. This can be translated as follows:

\begin{proposition}\label{prop-measurable-sum}
Let $(\mathfrak{T},\mathcal{T})$ be a measurable space with  $\mathcal{T}$ complete with respect to a $\sigma$-finite positive measure. Let $\Gamma_1:\mathfrak{T}\rightrightarrows Y$ be a $\mathcal{T}$-measurable multimapping with nonempty closed values and $\Gamma_2:\mathfrak{T}\rightrightarrows Y$ be a multimapping
with nonempty values whose graph $\mathrm{gph}\,\Gamma_2$ is $\mathcal{T}\otimes\mathcal{B}(Y)$ measurable.
Then for any $\mathcal{T}$-measurable selection $z(\cdot)$ of $\Gamma_1(\cdot)+\Gamma_2(\cdot)$ there exist
$\mathcal{T}$-measurable selections $z_1(\cdot)$ and $z_2(\cdot)$ of $\Gamma_1$ and $\Gamma_2$ respectively, such that
$$
    z(t)=z_1(t)+z_2(t) \quad\text{for all}\; t\in \mathfrak{T}.
$$
\end{proposition}

     Keep $(\mathfrak{T}, \mathcal{T})$ as a measurable space and suppose that
	$\mathcal{T}$ is complete with respect to  a $\sigma$-finite positive measure. Take a measurable multimapping $S: \mathfrak{T} \rightrightarrows H$ from $\mathfrak{T}$ into a separable Hilbert space $H$  such that all the values $S(t)$ are nonempty and prox-regular. Fix $x(\cdot): \mathfrak{T}\to H$ as a $\mathcal{T}$-measurable selection of $S(\cdot)$. For each $u \in H$ by \eqref{eq-DeriDist} one has
$$
    d'_{S(t)}(x(t);u)=\lim_{n\uparrow \infty}n\,d_{S(t)}\left(x(t)+\frac{1}{n}u \right),
$$
so $d_{S(t)}(y)$ being continuous in $y$ and $\mathcal{T}$-measurable in $t$ by Theorem \ref{theor-measurable}(a), we see that $d'_{S(t)}(x(t);u)$ is $\mathcal{T}$-measurable in $t$ and it is also
(Lipschitz) continuous in $u$. So, the function $(t,u)\mapsto d'_{S(t)}(x(t);u)$ is
$\mathcal{T}\otimes\mathcal{B}(H)$-measurable, which ensures that the multimapping $\Gamma: \mathfrak{T}\rightrightarrows H$,  with $\Gamma(t):=T_{S(t)}x(t)$, has its graph
$$
   \mathrm{gph}\,\Gamma=\{(t,u)\in \mathfrak{T}\times H: d'_{S(t)}(x(t);u)=0\}
$$
$\mathcal{T}\otimes\mathcal{B}(H)$-measurable. Theorem \ref{theor-measurable}(c) implies that the multimapping $\Gamma(\cdot)$ is $\mathcal{T}$-measurable. Theorem \ref{theor-measurable}(b) furnishes
a Castaing representation of $\Gamma$, i.e. a sequence $(\sigma_n)_{n\in \N}$ of $\mathcal{T}$-measurable mappings from $\mathfrak{T}$ into $H$ such that for each $t\in \mathfrak{T}$
$$
  \Gamma(t)=\mathrm{cl}_Y\{\sigma_n(t):n\in \mathbb{N}\},
$$
which gives by \eqref{eq-NormTang}
$$
   N_{S(t)}x(t)=\bigcap_{n\in \N}\{\varsigma\in H: \langle \varsigma, \sigma_n(t)\rangle \leq 0\}.
$$
Since each multimapping $t\mapsto \{\varsigma\in H: \langle \varsigma, \sigma_n(t)\rangle \leq 0\}$
is $\mathcal{T}$-measurable according to the $\mathcal{T}\otimes \mathcal{B}(H)$-measurability of
its graph and to Theorem \ref{theor-measurable}(c), we obtain that the multimapping
\begin{equation}\label{eq-MeasNorm}
    t \mapsto N_{S(t)}x(t)\;\,\text{is}\; \mathcal{T}-\text{measurable}
\end{equation}
according to Theorem \ref{theor-measurable}(e).

	\section{Main Results}\label{sect-main}
	Throughout the rest of the paper, $H$ will be a (real) separable Hilbert space. Given the interval $I:=[T_0,T]$, we consider the following nonconvex integro-differential sweeping process with initial condition
	$x_0\in C(T_0)$
	\begin{equation*}
	\!  \! 	
	\mathcal{P}_{f_{1},f_{2}}(x_0,F)\! \!  \begin{cases} -\dot{x}(t)\!\in\! N_{C(t)}x(t)\!+\!f_{1}(t,x(t))\!+\!\!
	\int_{T_{0}}^{t}\!f_{2}(t,s, x(s))ds\!+\!F(t, x(t)) \, \text{a.e.}\, t\!\in\! I\\
	x(T_0)=x_{0},
	\end{cases}
	\end{equation*}
		where $f_1: I\times H \to H$ and $f_2:I\times I\times H \to H$ are given mappings and
	$C: I \rightrightarrows H$ and $F:I\times H \rightrightarrows H$ are given multimappings.
	
	  Our objective in this section is to state and establish the main result of existence of solutions for the integro-differential sweeping process  $\big(\mathcal{P}_{f_{1},f_{2}}(x_0,F)\big)$ under appropriate hypotheses. As usual, given a closed bounded set $P$ in $\mathbb{R}^N$, the Lebesgue measure on $P$ will be denoted $\lambda$ and $L^{1}(P,H)$ will be the space of Bochner $\lambda$-integrable mappings from $P$ into $H$. The
	integral of any mapping on $P$ with values in $H$ in the paper will be the Bochner integral with respect
	to the Lebesgue measure. Similarly, $\mathcal{C}(P,H)$ will be the space of continuous mappings from $P$  into $H$ endowed with the norm $\|\cdot\|_{\infty}$ of uniform convergence on $P$.
	
	\subsection{Hypotheses on the data}\label{subsect-HypoData}
	
	In this subsection we collect the basic hypotheses that will be used throughout the paper. We keep $H$ as a separable real Hilbert space  and we consider the interval $I=[T_0, T]$. As in Section \ref{sect-prel} we denote by $B[x,\eta]$ the closed ball of $H$ with center $x$ and radius $\eta >0$ and by $\mathbb{B}$ the closed unit ball of $H$ centered at the origin. It will be convenient to put $\mathbb{R}_+:=[0,+\infty[$. \vspace{0,4cm}
	\par \textit{\textbf{$\mathcal{H}(C)$.}} We will take $ C: I\rightrightarrows H$ as a multimapping with nonempty and closed values, and  we will consider the following conditions :
	\begin{enumerate}
		\item [$(\mathcal{H}^C_{1})$] $C(t)$ varies in an absolutely continuous way, i.e. there exists an absolutely
		continuous function $\upsilon : I \rightarrow \mathbb{R}$, which is monotone increasing and
		$$
		d(y,C(t)) \leq d(y,C(s)) +|\upsilon(t)-\upsilon(s)| \;\text{for all}\;y\in H\;\text{and}\; s,t\in I,
		$$
		that is, if the sets $C(t)$ are bounded
		\begin{equation*}
		\mathrm{haus}(C(t), C(s)) \leq |\upsilon(t)-\upsilon(s)|\quad \text{for all}\; s, t\in I,
		\end{equation*}
		where $\mathrm{haus}(\cdot,\cdot)$ denotes the classical Hausdorff-Pompeiu distance;
		\item [$(\mathcal{H}^C_{2})$] there exists $R >0$ such that $C(t)$ is $R$-prox-regular for each $t\in I$.
	\end{enumerate}\vspace{0,4cm}
	\par \textit{\textbf{$\mathcal{H}(F)$.}} The multimapping $ F: I\times H\rightrightarrows H$ will be  with nonempty closed values and the following conditions will be required:
	\begin{enumerate}
		\item [$(\mathcal{H}_{0}^F)$] For every $x\in H$ $F(\cdot, x)$ is Lebesgue-measurable;
		\item [$(\mathcal{H}_{1}^F)$] there exists a non-negative function $k(\cdot)\in L^1(I,\mathbb{R}_+)$ satisfying the Lipschitz property
\begin{equation}\label{eq-Lips-F}
  F(t,y) \subset F(t,x) +k(t)\|x-y\|\mathbb{B} \quad\text{for all}\;t\in I\;\text{and}\; x,y\in H;
\end{equation}
		\item [$(\mathcal{H}_{2}^F)$] there exists a non-negative function $\gamma(\cdot)\in L^1(I,\mathbb{R}_+)$ with $F(t,x) \subset \gamma(t)\mathbb{B}$ for all $t\in I$ and $x\in H$.
	\end{enumerate}
	\vspace{0,4cm}
	\par \textit{\textbf{ $\mathcal{H}$($f_1$)}. } For the single-valued mapping $f_1: I\times H\rightarrow H$  we will consider the following conditions :
	\begin{enumerate}
		\item [$(\mathcal{H}_{0}^{f_1})$]$ f_{1} : I\times H \rightarrow H $ is Lebesgue measurable in time;
		\item [$(\mathcal{H}_{1}^{f_1})$]	there exists a non-negative function $ \beta_{1}(\cdot)\in L^{1}(I,\mathbb{R}_+) $ such that
		\begin{equation*}
		\lVert f_{1}(t,x)\rVert\leq \beta_{1}(t)(1+\lVert x \rVert) \,\,\,\text{for all}\;t\in
		I\; \text{and} \; x\in H;
		\end{equation*}
		\item [$(\mathcal{H}_{2}^{f_1})$] for each real $ \eta>0 $ there exists a non-negative function $L_{1}^{\eta}(\cdot)$ $\lambda$-integrable on $I$ and
		such that for any $ t\in I$ and for any $ (x,y)\in {B}[0,\eta]\times {B}[0,\eta] $,
		\begin{equation*}
		\lVert f_{1}(t,x)-f_{1}(t,y) \rVert\leq L_{1}^{\eta}(t)\lVert x-y \rVert.
		\end{equation*}	
	\end{enumerate}\vspace{0,4cm}
	\par \textit{\textbf{ $\mathcal{H}$($f_2$)}. } For the single-valued mapping $f_2: I^2\times H\rightarrow H$   we will consider the following conditions :
	\begin{enumerate}
		\item [$(\mathcal{H}_{0}^{f_2})$] $f_2(\cdot,\cdot,x)$ is $\mathcal{L}(I)\otimes\mathcal{L}(I)$  measurable on $I^2$ for each $x\in H$;
		\item [$(\mathcal{H}_{1}^{f_2})$] there exist non-negative  functions $ \beta_2(\cdot)\in L^{1}(I,\mathbb{R}_+) $ and  $ g(\cdot)\in L^{1}(Q_{\Delta},\mathbb{R}_+) $ such that
		\begin{equation*}
		\lVert f_{2}(t,s,x) \rVert\leq  g(t,s) +\beta_2(t)\lVert x \rVert \,\,\,\text{for any}\,\, (t,s)\in Q_{\Delta}\,\,\,\text{and any}\,\,\,x\in\underset{t\in I}\bigcup C(t).
		\end{equation*}
		\item [$(\mathcal{H}_{2}^{f_2})$] for each real $ \eta>0 $ there exists a non-negative function $L_{2}^{\eta}(\cdot)$ $\lambda$-integrable on $I$ and
		such that for all $ (t,s)\in Q_{\Delta} $ and $ (x,y)\in {B}[0,\eta]\times {B}[0,\eta]$,
		\begin{equation*}
		\lVert f_2(t,s,x)-f_2(t,s,y) \rVert\leq L_{2}^{\eta}(t)\lVert x-y \rVert.
		\end{equation*}	
		Above $Q_{\Delta}$ stands for the set
		$$
		Q_{\Delta}:=\{(t,s)\in I^{2}: s \leq t\}.
		$$
	\end{enumerate}
	
\begin{rem}\label{rem-Haus-Growth}
\em{
 (a)  It is worth noticing that, on the one hand, the condition $(\mathcal{H}_1^{F})$ readily ensures that for any $t\in I$ and any $x,y \in H$
$$
    |d_{F(t,x)}(w)-d_{F(t,y)}(w)| \leq k(t)\|x-y\| \quad\text{for every}\; w\in H,
$$
or equivalently
$$
   \mathrm{haus}(F(t,x),F(t,y))\leq k(t)\|x-y\|,
$$
and that, on the other hand, the latter condition entails with any $\varepsilon >0$ and $k_{\varepsilon}(t)=k(t)+\varepsilon$
$$
   F(t,y) \subset F(t,x) +k_{\varepsilon}(t)\|x-y\|\mathbb{B}.
$$
(b) It is also worth noticing that the condition $(\mathcal{H}_2^{F})$ gives $d(0,F(t,x))\leq \gamma(t)$
for all $t\in I$ and $x\in H$.}
\hfill $\square$
\end{rem}	

\subsection{Existence theorem}\label{subsect-ExisTheor}

    Before proceeding with the development of the proof of existence of solution of $\left(\mathcal{P}_{f_1,f_2}(x_0,F)\right)$, it is worth noticing the following. Assume that $x(\cdot):I \to H$ is an absolutely continuous solution of $\left(\mathcal{P}_{f_1,f_2}(x_0,F)\right)$ and that hypotheses $(\mathcal{H}(C))$, $(\mathcal{H}(F))$, $(\mathcal{H}(f_1))$ and $(\mathcal{H}(f_2))$ are satisfied. Observing that $(s,t)\mapsto f_2(t,s,x(s))$ is $\mathcal{L}(I)\otimes\mathcal{L}(I)$ measurable (by continuity of $x(\cdot)$) and writing
$$
   \int_{T_0}^tf_2(t,s,x(s))\,ds=\int_{T_0}^T\mathbf{1}_{Q_{\Delta}}(t,s)f_2(t,s,x(s)),
$$
the Fubini theorem ensures that the mapping $I\ni t\mapsto \int_{T_0}^tf_2(t,s,x(s))\,ds$ is $\mathcal{L}(I)$-measurable; above the function $\mathbf{1}_{Q_{\Delta}}$ is defined on $I\times{I}$ by $\mathbf{1}_{Q_{\Delta}}(t,s)=1$ if $(t,s)\in Q_{\Delta}$ and
$\mathbf{1}_{Q_{\Delta}}(t,s)=0$ otherwise.  This, \eqref{eq-MeasNorm} and Theorem \ref{theor-measurable}(b) tell us that the multimapping
$$
     I \ni t \mapsto M(t):=N_{C(t)}x(t) + f_1(t,x(t)) + \int_{T_0}^tf_2(t,s,x(s))\,ds
$$ 		
is $\mathcal{L}(I)$-measurable. On the other hand, for any $w\in H$ putting
$\delta_{w}(t,u):=d_{F(t,u)}(w)$ we have that $\delta_w(\cdot,u)$ is $\mathcal{L}(I)$-measurable by $(\mathcal{H}_0^F)$ and by Theorem \ref{theor-measurable}(a), and that $\delta_w(t,\cdot)$ is (Lipschitz) continuous on
$H$. Therefore, the function $I \ni t\mapsto d_{F(t,x(t))}(w)$ is $\mathcal{L}(I)$-measurable for every $w\in H$, so Theorem \ref{theor-measurable}(a) ensures that the multimapping $I\ni t\mapsto F(t,x(t))$ is $\mathcal{L}(I)$-measurable. The inclusion $-\dot{x}(t)\in M(t)+F(t,x(t))$ and both measurability properties above easily give by Proposition \ref{prop-measurable-sum} that there is an $\mathcal{L}(I)$-measurable mapping $z:I \to H$ with
$z(t) \in F(t,x(t))$ for every $t\in I$ and such that $-\dot{x}(t)\in M(t) + z(t)$ a.e. $t\in I$. So, for such a measurable selection $z(\cdot)$ of $t\mapsto F(t,x(t))$ we have
\begin{equation}\label{eq-ConsSol}
   \begin{cases}
		-\dot{x}(t) \in N_{C(t)}x(t) +f_1(t,x(t))+\int_{T_0}^tf_2(t,s,x(s))ds + z(t)\;\,\text{a.e.}\; t\in I\\
		x(T_0)=x_0,
	 \end{cases}
\end{equation}
and we observe by $(\mathcal{H}^F_2)$ that $z(t) \in \Gamma_{\gamma}(t):=\gamma(t)\mathbb{B}$.

    Taking the latter observation into account, we start with a result concerning differential inclusions where  suitable measurable selections
$z(\cdot)$ of $\Gamma_{\gamma}$ are involved instead of the multimapping $F(\cdot,\cdot)$.

\begin{theorem}\label{theor-Translate}
	Let $I:=[T_0,T]$ and let hypotheses $\mathcal{H}(C)$, $\mathcal{H}(f_1)$, $\mathcal{H}(f_2)$ hold. Let
	$\Gamma_{\gamma}(t):=\gamma(t)\mathbb{B}$ for all $t\in I$, where $\gamma(\cdot)$ is given in ($\mathcal{H}^F_2$).\\
(a)	For any $z\in \mathrm{Sel}\,\Gamma_{\gamma}$, for any initial point $x_{0} \in H$ with $x_{0}\in C(T_0)$  the integro-differential inclusion
\begin{equation*}
	\left(\mathcal{P}_{f_1, f_2}(x_0,z)\right) \! 	
  \begin{cases} -\dot{x}(t)\in N_{C(t)}x(t)\!+\! f_{1}(t,x(t)) \!+\! \int_{T_{0}}^{t}f_{2}(t,s, x(s))ds+z(t) \, \text{a.e.}\, t\!\in\! I\\
	x(T_0)=x_{0}
	\end{cases}
\end{equation*}
	has a unique absolutely  continuous solution  $\zeta_z :I\rightarrow H$. Further, this solution satisfies the inequalities:
	\begin{equation*}
	\lVert \zeta_z(t) \rVert \leq \eta \quad \text{for all}\; t\in I,
	\end{equation*}
	\begin{equation*}
	\lVert f_{1}(t,\zeta_z(t))+z(t) \rVert\leq (1+\eta)(\beta_1(t)+\|z(t)\|) \leq
	(1+\eta)(\beta_{1}+\gamma)(t) \quad\text{for all}\;t\in I,
	\end{equation*}
	\begin{equation*}
	\lVert f_{2}(t,s,\zeta_z(s)) \rVert \leq g(t,s) + \eta\beta_2(t) \quad \text{for all}\,\,\,(t,s)\in Q_{\Delta}, \label{42}
	\end{equation*}
	and for $\varphi(t)$ given for any $t\in I$ by
	$$
	    \varphi(t):= \rvert\dot{\upsilon}(t)\rvert+ (1+\eta)(\beta_{1}+\gamma)(t)+\int_{T_0}^{t}g(t,s)\,ds+\eta(T-T_{0})\beta_2(t),
	$$
	one has for almost every $t \in I$
	\begin{equation*}
	  \left\| \dot{x}(t)+f_{1}(t,\zeta_z(t))+z(t)+\int_{T_0}^{t}f_{2}(t,s,\zeta_z(s))\,ds
	  \right\| \leq \varphi(t),
	\end{equation*}
where $\eta$ is any fixed real satisfying
	$$
	   \eta  \geq \lVert x_{0} \rVert\,\Psi(T_0,T)
	 + \Psi(T_0,T)\int_{T_0}^{T}\bigg(\rvert\dot{\upsilon}(s)\rvert
	 +2(\beta_{1}+\gamma)(s)+2\int_{T_0}^{T}g(s,\tau)\,d\tau\bigg)\,ds,
	$$
	with $\Psi(T_0,T):=\exp\bigg(\int_{T_0}^{T}(b(\tau)+1)\,d\tau\bigg)$ and
	 $b(t):=2\max\{(\beta_{1}+\gamma)(t),\beta_2(t)\}$. \\
	(b)  Given $x^i_0\in C(T_0)$ and $z_i\in \mathrm{Sel}\,\Gamma_{\gamma}$, $i=1,2$, denoting $\xi_{z_i}$ the solution of $\big(\mathcal{P}_{f_1,f_2}(x^i_0,z_i)\big)$, i.e. of \eqref{eq-Sweep3} with initial conditions $\xi_{z_i}(T_0)=x_{0}^{i}$ and with the mapping $z$ in place of the multimapping $F$, for
	$\mu_0:=\max\{\|x^1_0\|,\|x^2_0\|\}$ and for any real $\eta >0$ such that
$$
	   \eta  \geq \mu_0\,\Psi(T_0,T)
	 + \Psi(T_0,T)\int_{T_0}^{T}\bigg(\rvert\dot{\upsilon}(s)\rvert
	 +2(\beta_{1}+\gamma)(s)+2\int_{T_0}^{T}g(s,\tau)\,d\tau\bigg)\,ds,
	$$	
	the solutions $\xi_{z_1}$ and $\xi_{z_2}$  satisfy for every $t\in I$ the inequality
	\begin{equation}\label{estim_1}
	\lVert \xi_{z_1}(t)-\xi_{z_2}(t)  \rVert \leq \Phi(T_0, T)\bigg( \lVert x_{0}^{1}-x_{0}^{2} \rVert +\int_{T_{0}}^{t}\|z_1(s)-z_2(s)\|ds\bigg),
	\end{equation}
	where\  $\Phi(T_0, T):=\exp\bigg(\int_{T_{0}}^{T}(K(\tau)+1)\,d\tau\bigg)$ and $ K(t):=\max\bigg\{L_{1}^{\eta}(t)+\dfrac{\varphi(t)}{r},L_{2}^{\eta}(t)\bigg\} $.\\
\end{theorem}
\begin{proof}
First, the existence of an absolutely continuous solution $\zeta_z(\cdot)$ and the properties in (a) for
$\big(\mathcal{P}_{f_1, f_2}(x_0, z)\big)$ are ensured by \cite[Theorem 4.2]{Bouach_Haddad_Thibault}.\\
Indeed, it is enough to put $f(t, x):=f_1(t, x)+z(t)$, so $f$ satisfies the assumption $\mathcal{H}(f_1)$ with $f$ in place of $f_1$ and satisfies also the inequality for all $t\in I$ and $x\in H$
\begin{equation*}
\|f(t, x)\|\leq (\beta_1(t)+\|z(t)\|)(t) (1+\|x\|)\leq (\beta_1+\gamma)(t)(1+\|x\|.
\end{equation*}

       Let us prove the estimation \eqref{estim_1}. Let $z_1$, $z_2$ be measurable selections of $\Gamma_{\gamma}$.
 Let $\xi_{z_{i}}(\cdot)$ be the solution of $\big(\mathcal{P}_{f_1,f_2}(x^i_0,z_i)\big)$), i.e. of \eqref{eq-Sweep3} with initial condition $x^i_0\in C(T_0)$ and with the mapping $z_i$ in place of the multimapping $F$, $i=1,2$. Then by the hypomonotonicity property of the normal cone in  Proposition \ref{prop-Hypo}, with
$J_i(t):=\int_{T_{0}}^{t} \! f_{2}(t,s,\xi_{z_i}(s))\,ds$ and $E_i(t):= f_{1}(t,\xi_{z_i}(t)) $ and with $\varphi(t)$ as given in (a) we have for almost every $t\in I$
\begin{equation*}
\begin{aligned}
& \Big\langle \! - \!\dot{\xi}_{z_1}(t) \! -z_1(t)- \! E_1(t)\! - \! J_1(t) + \! \dot{\xi}_{z_2}(t) \! +z_2(t)+ \! E_2(t) \!+ \! J_2(t), \xi_{z_1}(t) \! - \! \xi_{z_2}(t)\Big\rangle\\
&\leq \dfrac{\varphi(t)}{r}\lVert \xi_{z_1}(t)-\xi_{z_2}(t)  \rVert^{2},
\end{aligned}
\end{equation*}
from which we obtain
\begin{align*}
 &\quad\; \langle \dot{\xi}_{z_2}(t)-\dot{\xi}(z_1)(t), \xi_{z_2}(t)-\xi_{z_1}(t)\rangle \\
& \leq \dfrac{\varphi(t)}{r}\lVert \xi_{z_2}(t)-\xi_{z_1}(t)  \rVert^{2}
+ \langle E_1(t)- E_2(t) , \xi_{z_2}(t)-\xi_{z_1}(t) \rangle\\
&\ \ \ \ \ \  \ \ \ \ + \langle z_1(t)- z_2(t) , \xi_{z_2}(t)-\xi_{z_1}(t) \rangle+ \langle J_1(t)- J_2(t), \xi_{z_2}(t)-\xi_{z_1}(t) \rangle.
\end{align*}
Since, by the assumptions $ \mathcal{H}(f_1) $ and $ \mathcal{H}(f_2) $, there are with $\eta >0$ (as taken in the assumption in (b)) non-negative functions $L_{1}^{\eta}(\cdot)$ and $L_{2}^{\eta}(\cdot)$ in $ L^{1}(I,\mathbb{R}) $  such
that $f_{1}(t,\cdot)$ and $f_{2}(t,s,\cdot)$ are $L_{1}^{\eta}(t)$-Lipschitz and $L_{2}^{\eta}(t)$-Lipschitz respectively on $B[0,\eta] $, the above
inequality along with the inclusion $ \xi_{z_i}(t)\in B[0, \eta]$, $i=1,2$, entails with
$D_{1,2}(t):=\|\xi_{z_1}(t)-\xi_{z_2}(t)\|$ and $Z(t):=z_1(t)-z_2(t)$
\begin{equation*}
\dfrac{d}{dt}\big(D_{1,2}(t)^2\big) \!\leq\! 2\bigg(L^{\eta}_{1}(t) \! + \!  \dfrac{\varphi(t)}{r} \bigg)D_{1,2} (t)^{2}\! + 2\|Z(t)\|D_{1,2}(t)+\! 2L_{2}^{\eta}(t)D_{1,2}(t)\!\!\int_{T_{0}}^{t}\!\!D_{1,2}(s)\,ds.
\end{equation*}
Applying the Gronwall-like differential inequality in Lemma \ref{lem-Grown2}, it results that
\begin{equation*}
\begin{aligned}
\lVert \xi_{z_1}(t)-\xi_{z_2}(t)  \rVert &\leq \lVert x_{0}^{1}-x_{0}^{2} \rVert \exp\big(\int_{T_{0}}^{t}(K(\tau)+1)\,d\tau\big)\\
  & \quad\; +\int_{T_{0}}^{t}\|z_1(s)-z_2(s)\|\exp\big(\int_{s}^{t}(K(\tau)+1)\,d\tau\big)ds\\
	&\leq \Phi(T_0,T) \big(\lVert x_{0}^{1}-x_{0}^{2} \rVert+ \int_{T_{0}}^{t}\|z_1(s)-z_2(s)\|ds\big),
\end{aligned}
\end{equation*}
where $ K(t):=\max\bigg\{L_{1}^{\eta}(t)+\dfrac{\varphi(t)}{r},L_{2}^{\eta}(t)\bigg\} $ and
$ \Phi(T_0,T):=\exp\big(\int_{T_0}^{T}(K(\tau)+1)\,d\tau\big)$. The proof is then complete.
\end{proof}

\vspace{0.4cm}
  Assume that $(\mathcal{H}(C))$, $(\mathcal{H}(F))$, $(\mathcal{H}(f_1))$ and $(\mathcal{H}(f_2))$ holds for
	$C$, $F$, $f_1$ and $f_2$ as presented in the beginning of this section and fix $x_0\in C(T_0)$. Let $\gamma:I\to \mathbb{R_+}$ be $\lambda$-integrable on $I$ satisfying $(\mathcal{H}^F_2)$, i.e. $F(t,x)\subset \gamma(t)\mathbb{B}$ for all
	$t\in I$ and $x\in H$. The multimapping $\Gamma_{\gamma}(\cdot):=\gamma(\cdot)\mathbb{B}$
	is clearly measurable. Each $z\in \mathrm{Sel}\,\Gamma_{\gamma}$
	is in $L^1(I,H)$ and we keep $\zeta_z(\cdot)$ as the notation for the absolutely continuous solution of
	the integro-differential sweeping process
	\begin{equation}\label{eq-Diff-z}
	  \begin{cases}
		-\dot{\zeta}_z(t) \in N_{C(t)}\zeta_z(t) + f_1(t,\zeta_z(t))
		+\int_{T_0}^tf_2(t,s,\zeta_z(s))\,ds+z(t) \;\,
		\text{a.e.}\;t\in I \\
		\zeta_z(T_0)=x_0.
		\end{cases}
	\end{equation}
	Fix $q_0\in C(T_0)$ and take  $\mu_0:= \max\{\|x_0\|,\|q_0\|\}$ and any real $\eta >0$
		with
	$$
	   \eta  \geq \mu_0\,\Psi(T_0,T)
	 + \Psi(T_0,T)\int_{T_0}^{T}\bigg(\rvert\dot{\upsilon}(s)\rvert
	 +2(\beta_{1}+\gamma)(s)+2\int_{T_0}^{T}g(s,\tau)\,d\tau\bigg)\,ds,
	$$
	and let $q:I\to H$ be the absolutely continuous solution of
	\begin{equation}\label{eq-Diff-q}
	  \begin{cases}
	-\dot{q}(t) \in N_{C(t)}q(t)+f_1(t,q(t))+\int_{T_0}^tf_2(t,s,q(s))\,ds \;\, \text{a.e.}\; t\in I\\
		q(0)=q_0.
		\end{cases}
	\end{equation}	
	By Theorem \ref{theor-Translate}(b) for each $z\in \mathrm{Sel}\,\Gamma_{\gamma}$ one has for every $t\in I$
	\begin{equation}\label{eq-Est-z-q}
	   \|\zeta_z(t)-q(t)\| \leq \Phi(T_0,T)\left(\|x_0-q_0\|+\int_{T_0}^t\|z(s)\|\,ds\right),
	\end{equation}
	where $\Phi(T_0,T)$ is as given in (b) of Theorem \ref{theor-Translate}.
	
	    Take a real $r_0 \geq \|x_0-q_0\|$ and for $k:I\to \mathbb{R}_+$ in $L^1(I,\mathbb{R})$ as given by
	$(\mathcal{H}^F_1)$, let $r:I\to \mathbb{R}$ be the solution of the differential
	equation on $I$
	\begin{equation}\label{eq-Diff-r}
	   \begin{cases}
		  \dot{r}(t)=\gamma(t)+ \Phi(T_0,T)k(t)r(t) \\
			r(T_0)=r_0.
		 \end{cases}
	\end{equation}
	The solution of the latter differential equation is furnished by
\begin{equation}\label{eq-Expr-r}
   r(t)=r_{0}\, \exp(c(t))+\int_{T_0}^{t}\exp(c(t)-c(s))\gamma(s)\,ds \quad\text{for all}\; t\in I,
\end{equation}
where $c(\cdot)$ is the absolutely continuous function on $I$ given
\begin{equation}\label{eq-Exp-c(.)}
 c(t):=\Phi(T_0, T)\int_{T_0}^{t}k(\tau)d\tau \quad\text{for all}\;  t\in I,
\end{equation}
so we see that $r(\cdot)$ is absolutely continuous on $I$ and it can be verified that it satisfies
the differential equation \eqref{eq-Diff-r}. With the latter expression of $r(\cdot)$ at hands, let $N$ be a Borel subset of $I$ with null measure such that
\eqref{eq-Diff-r} holds for all $t\in I\setminus N$ and change or define $\dot{r}(\cdot)$ on $N$ in putting $\dot{r}(t)= \gamma(t)+ \Phi(T_0,T)k(t)r(t)$ for all $t\in N$. Doing so, the equation \eqref{eq-Diff-r} is
now satisfied for all $t\in I$. Note also that for all $t\in I$
\begin{equation}\label{eq-MajDer-r}
  r(t) \geq r_0 \quad\text{and} \quad
	\dot{r}(t) \geq \gamma(t).
\end{equation}

      Now, we are able to prove the main theorem of existence of solution of the differential inclusion
	$(\mathcal{P}_{f_1,f_2}(x_0,F))$. Various parts of the development of the proof are inspired by A.F.  Filippov \cite{Fili} and A.A. Tolstonogov \cite{Tolst2017}, and also A.D. Ioffe \cite{Ioff}.

   We keep as above  $\Gamma_{\gamma}(t):=\gamma(t)\mathbb{B}$ for all $t\in I$, and for each
 $z\in \mathrm{Sel}\,\Gamma_{\gamma}$ we take any solution $\zeta_z(\cdot)$ of $(\mathcal{P}_{f_1,f_2,z})$ with initial condition $\zeta_z(T_0)=x_0$.

\begin{theorem}\label{theor-main}
	Let $I:=[T_0,T]$ and let hypotheses $\mathcal{H}(C)$, $\mathcal{H}(F)$, $ \mathcal{H}(f_{1}),$ $\mathcal{H}(f_2)$ hold. Let $\gamma(\cdot)\in L^1(I,\mathbb{R}_+)$ be given by $(\mathcal{H}^F_2)$. Let any $q_0\in C(T_0)$ and let $q(\cdot)$ be any absolutely continuous solution of \eqref{eq-Diff-q}. Then, for any $x_{0}\in C(T_0)$ there exists some $z\in \mathrm{Sel}\,\Gamma_{\gamma}$ such that for the absolutely continuous solution $\zeta_z(\cdot)$ of the differential inclusion $(\mathcal{P}_{f_1,f_2}(x_0,z))$ with initial condition $\zeta_z(T_0)=x_0$ one has
$$
    z(t) \in F(t,\zeta_z(t)) \,\,\, \text{a.e.}\; t\in I
$$
as well as the estimations
	\begin{equation}\label{estim_final}
	\Vert \zeta_z(t)- q(t)\Vert \leq  \Phi(T_0, T)r(t) \quad\text{and}\quad \Vert z(t) \Vert \leq \dot{r}(t),
	\end{equation}
where $r(\cdot)$ and $\Phi(T_0,T)$ are as given above in \eqref{eq-Diff-r} and Theorem \ref{theor-Translate}(b) respectively; in particular, this absolutely continuous mapping $\zeta_z:I\to H$ is a solution of the differential inclusion $(\mathcal{P}_{f_1,f_2}(x_0,F))$
\begin{equation*}
	  \begin{cases}
		-\dot{x}(t) \in N_{C(t)}x(t)+f_1(t,x(t))+\int_{T_0}^tf_2(t,s,x(s))\,ds+F(t,x(t)) \;\,
		\text{a.e.}\;t\in I \\
		x(T_0)=x_0.
		\end{cases}
\end{equation*}
\end{theorem}
\begin{proof}
	The proof of existence of solution proceeds with two steps.\\
Let $G: I\times H \rightrightarrows H$ be the multimapping defined by
	\begin{equation}\label{G}
	G(t, x)=F(t, x+q(t)) \quad\text{for all}\; x\in H, \ t\in I.
	\end{equation}
For any $\mathcal{L}(I)$-measurable mapping $p:I\to H$ put $M_p(t,x):=F(t,x+p(t))$, so $G(t,x)=M_q(t,x)$.
 Fix $x,y\in H$ and consider the function $\theta:I\times H\to \R$ with $\theta(t,u)=d(y,F(t,x+u))$ for all
$t\in I$ and $u\in H$. For each $u\in H$ the function $\theta(\cdot,u)$ is $\mathcal{L}(I)$-measurable by Theorem \ref{theor-measurable}(a) and for each $t\in I$ the function $\theta(t,\cdot)$ is continuous, even
Lipschitz, on $H$. This ensures that $\theta$ is a Carath\'eodory function, hence $\mathcal{L}(I)\otimes\mathcal{B}(H)$ measurable. The mapping
$t\mapsto (t,x+p(t))$ from $I$ into $I\times H$ being $(\mathcal{L}(I),\mathcal{L}(I)\otimes \mathcal{B}(H))$ measurable, it ensues that the function $t\mapsto \theta(t,x+p(t))$ is $\mathcal{L}(I)$-measurable. This means that the function $t\mapsto d(y,M_p(t,x))$ is $\mathcal{L}(I)$-measurable, so by Theorem
\ref{theor-measurable}(a) again the multimapping $t\mapsto M_p(t,x)$ is $\mathcal{L}(I)$-measurable. \\
$\bullet$ By what precedes the multimapping $t\mapsto G(t,x)$ is $\mathcal{L}(I)$-measurable. \\
$\bullet$ Clearly for all $x,y\in H$ one has
$$
    G(t,y) \subset G(t,x)+ k(t)\|x-y\|\mathbb{B}\;\;\text{and}\;\;
		G(t,x) \subset \gamma(t)\mathbb{B}.
$$
\textbf{Step 1. Construction by induction of sequences  $(y_{i})_{i}$, $(z_i)_{i}$.} \newline
We construct by induction measurable sequences $(y_{i})_i$, $(z_{i})_i$,  $i\in \{0\}\cup\mathbb{N}$ with the properties
\begin{equation}\label{properity_1_of_y}
y_{0}(t)=0,\; z_i(t)\in G(t,y_i(t)), \;\,  y_{i+1}(t)=\zeta_{z_{i}}(t)- q(t),
\end{equation}
and such that for every $t\in I$
\begin{equation}\label{properity_3_of_y}
\Vert z_i(t)\Vert \leq \dot{r}(t), \;\, \Vert y_{i+1}(t)\Vert \leq \Phi(T_0, T)r(t),
\end{equation}
\begin{equation}\label{properity_4_of_y}
\Vert z_i(t)-z_{i-1}(t)\Vert \leq k(t)\Vert y_i(t)-y_{i-1}(t)\Vert,  \;\,  i\geq 1,
\end{equation}
\begin{equation}\label{properity_5_of_y}
\Vert z_i(t)-z_{i-1}(t)\Vert \leq \Phi(T_0, T)k(t)\bigg\{ r_0\dfrac{[c(t)]^{i-1}}{(i-1)!}+\int_{T_0}^{t} \dfrac{[c(t)-c(s)]^{i-1}}{(i-1)!} \gamma(s) ds\bigg\}, \; i\geq 1,
\end{equation}
\begin{equation}\label{properity_6_of_y}
\Vert y_{i+1}(t)-y_{i}(t)\Vert \leq \Phi(T_0, T)\bigg\{ r_0\dfrac{[c(t)]^{i}}{i!}+\int_{T_0}^{t} \dfrac{[c(t)-c(s)]^{i}}{i!} \gamma(s) ds\bigg\},
\end{equation}
where $c: I\to \mathbb{R}_+$ is the function defined in \eqref{eq-Exp-c(.)}.
\vskip 0.2cm
 \textbf{For i=0.} \quad
 By Theorem \ref{theor-measurable}(d) we can choose an $\mathcal{L}(I)$-measurable selection $z_0$ of the measurable multimapping  $t\mapsto G(t,y_0(t))$, so $z_0(t) \in G(t,y_0(t))$ and $\|z_0(t)\|\leq \gamma(t)\leq \dot{r}(t)$. Noting that $z_0\in \mathrm{Sel}\,\Gamma_{\gamma}$, we take $\zeta_{z_0}(\cdot)$ as the absolutely continuous solution of
$(\mathcal{P}_{f_1,f_2}(x_0,z_0))$ and we put $y_1(t)=\zeta_{z_0}(t)-q(t)$ for all $t\in I$.
By Theorem \ref{theor-Translate}(b) we have (since $r_0 \geq \|x_0-q_0\|$)
\begin{align*}
   \|y_1(t)-y_0(t)\| & \leq \Phi(T_0,T)\left(\|x_0-q_0\| +\int_{T_0}^t\|z_0(s)\|\,ds\right) \\
	 & \leq \Phi(T_0,T)\left\{r_0\frac{[c(t)]^0}{0!} +\int_{T_0}^t\frac{[c(t)-c(s)]^0}{0!}\gamma(s)\,ds\right\},
\end{align*}
and also (since $y_0(t)=0$)
\begin{align*}
  \|y_1(t)\| & \leq \Phi(T_0,T)\left(\|x_0-q_0\| +\int_{T_0}^t\|z_0(s)\|\,ds\right) \\
	 & \leq \Phi(T_0,T)\left(r_0 +\int_{T_0}^t\dot{r}(s)\,ds\right)=\Phi(T_0,T)r(t).
\end{align*}
All the properties of the induction hold true for $i=0$.

\vskip 0.2cm
\textbf{For i=1.}\quad
Writing
$$
   z_0(t) \in G(t,y_0(t)) \subset G(t,y_1(t))+k(t)\|y_1(t)-y_0(t)\|\mathbb{B},
$$
applying Proposition \ref{prop-measurable-sum} and using, by what precedes, the $\mathcal{L}(I)$- measurability of the multimapping
$t\mapsto G(t,y_j(t))$, $j=0,1$, we can find an $\mathcal{L}(I)$-measurable selection $z_1(\cdot)$ of $t\mapsto G(t,y_1(t))$ such that $z_0(t)-z_1(t) \in k(t)\|y_1(t)-y_0(t)\|\mathbb{B}$, hence
\begin{equation}\label{eq-Diff-z1-z0}
    \|z_1(t)-z_0(t)\| \leq k(t)\|y_1(t)-y_0(t)\| \quad\text{for all}\; t\in I.
\end{equation}
Further, the inclusion $z_1(t) \in G(t,y_1(t))=F(t,y_1(t)+q(t))$ gives
$$
     \|z_1(t)\| \leq \gamma(t) \leq \dot{r}(t) \quad \text{for all}\; t\in I.
$$
Let $\zeta_{z_1}$ be the absolutely continuous solution of $(\mathcal{P}_{f_1,f_2}(x_0,z_1))$.
By \eqref{eq-Diff-z1-z0} and by what obtained in the step $i=0$ we have for all $t\in I$
\begin{align*}
   \|z_1(t)-z_0(t)\| & \leq k(t)\|y_1(t)\|=k(t)\|\zeta_{z_1}(t)-q(t)\| \\
    & \leq k(t)\Phi(T_0,T)\left(\|x_0-q_0\| +\int_{T_0}^t\|z_0(s)\|\,ds\right) \\
		 & \leq k(t)\Phi(T_0,T)\left(r_0 +\int_{T_0}^t\gamma(s)\,ds\right) \\
 & =\Phi(T_0,T)k(t)\left\{r_0\frac{[c(t)]^0}{0!}+\int_{T_0}^t\frac{[c(t)-c(s)]^0}{0!}\gamma(s)\,ds\right\}.
\end{align*}
For each $t\in I$ let $y_2(t)=\zeta_{z_1}(t)- q(t)$, so $y_2(t)-y_1(t)=\zeta_{z_1}(t)-\zeta_{z_0}(t)$, hence Theorem \ref{theor-Translate}(b) and the latter inequality above yield
\begin{align*}
   \|y_2(t)-y_1(t)\| & \leq \Phi(T_0,T)\left(\|x_0-x_0\|+\int_{T_0}^t\|z_1(s)-z_0(s)\|\,ds\right) \\
	  & \leq \Phi(T_0,T)\left\{\Phi(T_0,T)\int_{T_0}^tk(s)\left(r_0+\int_{T_0}^{s}\gamma(\tau)\,d\tau\right)
		        \,ds\right\}\\
		& = \phi(T_0,T)\left\{r_0c(t)+
		  \Phi(T_0,T)\int_{T_0}^tk(s)\left(\int_{T_0}^{s}\gamma(\tau)\,d\tau\right)\,ds\right\},
\end{align*}
where the latter equality follows from the fact that $\Phi(T_0,T)k(\cdot)$ is the derivative of the function $c(\cdot)$ according to the definition of $c(\cdot)$ in \eqref{eq-Exp-c(.)}. This gives according to Lemma \ref{lem-Expr-Int-c} below
\begin{equation*}			
	\|y_2(t)-y_1(t)\| 		
			\leq \Phi(T_0,T)\left\{r_0c(t)+\int_{T_0}^t[c(t)-c(s)]\gamma(s)\,ds\right\}.
\end{equation*}
It ensues that for each $t\in I$
$$
  \|y_2(t)-y_1(t)\|  \leq \Phi(T_0,T)\left\{r_0\frac{[c(t)]^1}{1!}
	  +\int_{T_0}^t\frac{[c(t)-c(s)]^1}{1!}\gamma(s)\,ds\right\}.
$$
On the other hand, for each $t\in I$ we also have by Theorem \ref{theor-Translate}(b) again
\begin{align*}
 \|y_2(t)\| & =\|\zeta_{z_1}(t)-q(t)\| \\
 & \leq \Phi(T_0,T)\left(\|x_0-q_0\| +\int_{T_0}^t\|z_1(s)\|\,ds\right) \\
 & \leq \Phi(T_0,T)\left(r_0 + \int_{T_0}^t\dot{r}(s)\,ds\right)=\phi(T_0,T)r(t).
\end{align*}
All the properties of the induction then hold true for $i=0$ and $i=1$.

\vskip 0.2cm
{\bf From $i-1$ to $i$.} \quad
Suppose that the construction has been done for $j=0,\cdots,i-1$.  Since $z_{i-1}\in G(t,y_{i-1}(t))$,
the assumption $(\mathcal{H}^F_1)$ gives for each $t \in I$
$$
  z_{i-1}(t) \in G(t,y_i(t)) +k(t)\|y_i(t)-y_{i-1}(t)\|\mathbb{B},
$$
so by Proposition \ref{prop-measurable-sum} there exists an $\mathcal{L}(I)$-measurable selection $z_i$ of the multimapping $t\mapsto G(t,y_i(t))$ such that
\begin{equation}\label{eq-Diff-z-Diff-y}
  \|z_i(t)-z_{i-1}(t)\|\leq k(t)\|y_i(t)-y_{i-1}(t)\| \,\,\,\text{for all}\; t\in I.
\end{equation}
Note also that the inclusion $z_i(t)\in G(t,y_i(t))$ ensures that $z_i\in \mathrm{Sel}\,\Gamma_{\gamma}$ and
$$
   \|z_i(t)\| \leq \gamma(t) \leq \dot{r}(t) \quad\text{for all}\; t\in I.
$$
Let $\zeta_{z_i}$ be the solution of $(\mathcal{P}_{f_1,f_2}(x_0,z_i))$.
Define $y_{i+1}(t)=\zeta_{z_{i}}- q(t)$ and note that $y_{i+1}(\cdot)$ is $\mathcal{L}(I)$-measurable.
By Theorem \ref{theor-Translate}(b) we have for every $t\in I$
\begin{align*}
\Vert  y_{i+1}(t)\Vert & = \Vert  \zeta_{z_i}(t)-q(t)\Vert
   \leq \Phi(T_0, T)\left(\|x_0-q_0\|+\int_{T_0}^t\|z_i(s)\|\,ds \right) \\
	  & \leq \Phi(T_0,T)\left(r_0+\int_{T_0}^t\dot{r}(s)\,ds \right)=\Phi(T_0,T)r(t).
\end{align*}
Using \eqref{eq-Diff-z-Diff-y} and \eqref{properity_6_of_y} fulfilled for $i-1$ we obtain for all $t\in I$
\begin{equation*}
 \Vert z_i(t)-z_{i-1}(t)\Vert \leq \Phi(T_0, T)k(t)\bigg[r_0\dfrac{[c(t)]^{i-1}}{(i-1)!}+\int_{T_0}^{t} \dfrac{[c(t)-c(s)]^{i-1}}{(i-1)!} \gamma(s) ds\bigg] .
\end{equation*}
Noting that $\Vert y_{i+1}(t) -y_{i}(t) \Vert =\Vert \zeta_{z_i}(t) -\zeta_{z_{i-1}}(t) \Vert$ by definition of
$y_i(\cdot)$ and $y_{i+1}(\cdot)$, Theorem \ref{theor-Translate}(b) and the latter inequality assures us that
\begin{align*}
 & \quad\; \Vert y_{i+1}(t) -y_{i}(t) \Vert \\
& \leq \Phi(T_0, T)\bigg(\Vert x_{0}-x _{0}\Vert+\int_{T_{0}}^{t} \Vert z_i(s)-z_{i-1}(s)\Vert ds\bigg)\\
  &=\Phi(T_0, T)\int_{T_{0}}^{t} \Vert z_i(s)-z_{i-1}(s)\Vert ds\\
	&\leq\Phi(T_0, T)\int_{T_{0}}^{t}\Phi(T_0, T)k(s)\bigg[r_0\dfrac{[c(s)]^{i-1}}{(i-1)!}+\int_{T_0}^{s} \dfrac{[c(s)-c(\tau)]^{i-1}}{(i-1)!} \gamma(\tau) d\tau\bigg]ds,
\end{align*}
which gives
\begin{align*}
	& \quad\; \Vert y_{i+1}(t) -y_{i}(t) \Vert \\
	&\leq \Phi(T_0,T)\left\{\int_{T_{0}}^{t}r_0\Phi(T_0, T) k(s)\dfrac{[c(s)]^{i-1}}{(i-1)!}ds+\int_{T_{0}}^{t}\Phi(T_0, T)k(s)\int_{T_0}^{s} \dfrac{[c(s)-c(\tau)]^{i-1}}{(i-1)!} \gamma(\tau) d\tau ds\right\} \\
	&=\Phi(T_0, T)\left\{r_0\dfrac{[c(s)]^{i}}{(i)!}+\int_{T_0}^{t} \dfrac{[c(t)-c(s)]^{i}}{(i)!} \gamma(s)  ds\right\},
\end{align*}
the latter equality being due to the fact that $ \Phi(T_0, T)k(t)$ is the derivative of $c(t):=\Phi(T_0, T)\int_{T_0}^{t}k(s)ds$ and to Lemma \ref{lem-Expr-Int-c} below. The sequences $(y_i)_i$ and $(z_i)_i$ are then well-defined by induction to satisfy the required properties.

\vspace{0.2cm}
\textbf{Step 2. Convergence of the sequence $y_{i}(\cdot)_i$.} \\
From \eqref{properity_5_of_y} we infer for each $t\in I$ that the series $ \sum_{i=0}^{\infty} \|z_i(t)-z_{i-1}(t)\|$ converges.
Therefore, $(z_i(t))_i$ is a Cauchy sequence, hence the sequence $(z_i(\cdot))_i $ converges pointwise to a measurable mapping $z(\cdot)$. Since $\|z_i(t)\|\leq \gamma(t)$ for all $t\in I$, the Lebesgue dominated convergence theorem entails that the sequence $(z_i(\cdot))_i$ converges to $z(\cdot)$ in the space
$L^{1}(I, H)$. Further, $\|z(t)\| \leq \gamma(t)$ for all $t\in I$, so $z\in \mathrm{Sel}\,\Gamma_\gamma$.
Let $\zeta_z(\cdot)$ be the  solution of $(\mathcal{P}_{f_1,f_2}(x_0,z)))$.
 Then, from the inequality in Theorem \ref{theor-Translate}(b) it follows that
\begin{equation}
\|\zeta_z(t)-\zeta_{z_i}(t)\|\leq \Phi(T_0, T)\int_{T_0}^{t}\|z(\tau)-z_i(\tau)\|\,d\tau .
\end{equation}
From this inequality and from the Lebesgue dominated convergence theorem we deduce that the sequence
$(\zeta_{z_{i}}(\cdot))_i$ converges to $\zeta_z(\cdot)$ in the space $C(I, H)$.
This and the equality
$y_{i+1}(\cdot) =\zeta_{z_{i}}(\cdot)- q(\cdot)$  furnish that the sequence $(y_i(\cdot))_i$ converges uniformly on $I$ to $y(\cdot):= \zeta_z(\cdot)- q(\cdot)$.
On the other hand, by $(\mathcal{H}^F_1)$ we have for every $t \in I$
\begin{equation*}
d(z_i(t), G(t, y(t)))\leq k(t)\|y_{i}(t)-y(t)\|,
\end{equation*}
so passing to the limit as $ i \rightarrow \infty $ in the latter inequality and taking \eqref{G} into account we derive that for each $t\in I$
\begin{equation}\label{zzz}
  z(t)\in G(t, y(t))=G(t, \zeta_z(t)-q(t))=F(t, \zeta_z(t)).
\end{equation}
Consequently, we conclude  that $\zeta_z(\cdot)$ is a solution of $(\mathcal{P}_{f_{1},f_{2}}(x_0,F))$ and the inequalities \eqref{estim_final} hold. The proof of the theorem is finished.
\end{proof}

\begin{lemma}\label{lem-Expr-Int-c}
Under notation above one has
$$
  \int_{T_{0}}^{t}\Phi(T_0, T)k(s)\int_{T_0}^{s} \dfrac{[c(s)-c(\tau)]^{i-1}}{(i-1)!} \gamma(\tau) d\tau ds
  = \int_{T_0}^t\dfrac{[c(t)-c(s)]^{i}}{i!} \gamma(s)ds.
$$
\end{lemma}
\begin{proof}
Put $k_0(t):=\Phi(T_0,T)k(t)$, so ${k}_0(t)=\dot{c}(t)$ since $c(t)=\Phi(T_0,T)\int_{T_0}^tk(s)ds$. Denoting $I(t)$ the integral in the left-hand side of the statement of the lemma, we have
$I(t)=I_1(t) - I_2(t)$ for every $t\in I$, where
$$
  I_1(t)=\int_{T_0}^tk_0(s)(\int_{T_0}^t\frac{[c(s)-c(\tau)]^{i-1}}{(i-1)!}\gamma(\tau)d\tau)ds,
$$
$$
	I_2(t)=\int_{T_0}^tk_0(s)(\int_{s}^t\frac{[c(s)-c(\tau)]^{i-1}}{(i-1)!}\gamma(\tau)d\tau)ds.
$$
Let us write by Fubini theorem
\begin{align*}
I_1(t) & = \int_{T_0}^t\left(\int_{T_0}^tk_0(s)\frac{[c(s)-c(\tau)]^{i-1}}{(i-1)!}ds\right)\gamma(\tau)d\tau\\
 &= \int_{T_0}^t\frac{[c(t)-c(\tau)]^{i}}{i!}\gamma(\tau)d\tau
    - \int_{T_0}^t\frac{[c(T_0)-c(\tau)]^{i}}{i!}\gamma(\tau)d\tau.
\end{align*}
Using again the Fubini theorem, we also have
\begin{align*}
I_2(t) & = \int_{T_0}^t\left(\int_s^tk_0(s)\frac{[c(s)-c(\tau)]^{i-1}}{(i-1)!}\gamma(\tau)d\tau\right)ds \\
 &= \int_{T_0}^t\left(\int_{T_0}^{\tau}k_0(s)\frac{[c(s)-c(\tau)]^{i-1}}{(i-1)!}ds\right)\gamma(\tau)d\tau\\
&= -\int_{T_0}^t\frac{[c(T_0)-c(\tau)]^i}{i!}\gamma(\tau)d\tau.
\end{align*}
Coming back to the equality $I(t)=I_1(t)-I_2(t)$, we deduce the desired
$$
   I(t)=\int_{T_0}^t\frac{[c(t)-c(\tau)]^i}{i!}\gamma(\tau)d\tau.
$$
\end{proof}

\begin{rem}\label{rem-f1-F}
 {\em We emphasize the presence (instead of the unique multimapping $F(t,\cdot)$) of both the mapping $f_1(t,\cdot)$ and the multimapping $F(t,\cdot)$. This is explained as follows. First, the condition $(\mathcal{H}^F_2)$ requires the boundedness of $F(t,\cdot)$ as $F(t,x)\subset \gamma(t)\mathbb{B}$ for all $x\in H$ while the mapping
$f_1(t,\cdot)$ has to satisfy merely the growth condition $\|f_1(t,x)\|\leq \beta_1(t)(1+\|x\|)$ for all
$x\in H$, so the multimapping $f_1(t,\cdot)+F(t,\cdot)$ does not  fulfill the  boundedness condition of type $(\mathcal{H}^F_2)$. On the other hand, $(\mathcal{H}^F_1)$ assumes the $L_1^{\eta}(t)$-Lipschitz condition for $f_1(t,\cdot)$ only on the ball $\eta\mathbb{B}$, hence the multimapping $f_1(t,\cdot)+F(t,\cdot)$ is not Lipschitz on the whole space $H$. Then, the conditions needed in Theorem \ref{theor-main} for the perturbation multimapping does not hold for $f_1(t,\cdot)+F(t,\cdot)$.}
\hfill $\square$
\end{rem}

   As obtained, the result in Theorem \ref{theor-main} allows to derive a relaxation property. Denote
$\overline{\mathrm{co}}\,F:I\times{H} \rightrightarrows H$ the multimapping defined by
$$
   (\overline{\mathrm{co}}\,F)(t,x):= \overline{\mathrm{co}}\,F(t,x) \quad\text{for all}\; t\in I,\;
	  x \in H,
$$
and note by $(\mathcal{H}^F_0)$ and Theorem \ref{theor-measurable}(b) that, for each $x\in H$, the multimapping $\overline{\mathrm{co}}\,F(\cdot,x)$ is $\mathcal{L}(I)$-measurable.
Fix $x_0\in C(T_0)$ and assume that the hypotheses $(\mathcal{H}^F_i)$, $i=0,1,2$, are satisfied. Noticing that
$ \overline{\mathrm{co}}\,F(t,x)+k(t)\|x-y\|\mathbb{B}$ is weakly closed and convex (the closedness property being due to the weak compactness of the unit ball $\mathbb{B}$ of $H$) we see by $(\mathcal{H}^F_1)$ that
$$
   \overline{\mathrm{co}}\,F(t,y) \subset \overline{\mathrm{co}}\,F(t,x) + k(t)\|x-y\|\mathbb{B},
$$
so all the hypotheses $(\mathcal{H}_i^{\overline{\mathrm{co}}\,F})$, with $i=0,1,2$, are fulfilled. Therefore, under the hypotheses $(\mathcal{H}(C))$, $(\mathcal{H}(F))$, $(\mathcal{H}(f_1))$ and
$(\mathcal{H}(f_2))$, the result in Theorem \ref{theor-main} also holds true with
$\overline{\mathrm{co}}\,F$ in place of $F$. Take (by Theorem \ref{theor-main} with
$\overline{\mathrm{co}}\,F$ in place of $F$) any absolutely continuous solution $y(\cdot):I \to H$ of
$(\mathcal{P}_{f_1,f_2}(x_0,\overline{\mathrm{co}}\,F))$ with initial condition $y(T_0)=x_0$, i.e.,
$$
  \begin{cases}
	  -\dot{y}(t) \in N_{C(t)}y(t)+f_1(t,y(t))+\int_{T_0}^tf_2(t,s,y(s))\,ds+
		\overline{\mathrm{co}}\,F(t,y(t)) \;\,
		\text{a.e.}\;t\in I \\
		y(T_0)=x_0.
	\end{cases}
$$
By the analysis yielding to \eqref{eq-ConsSol} and by Theorem \ref{theor-main} there is some Lebesgue  measurable selection $z: I \to H$ of $t\mapsto \overline{\mathrm{co}}\,F(t,y(t))$ such $y(\cdot)$ is an absolutely continuous solution of
$(\mathcal{P}_{f_1,f_2}(x_0,z))$, i.e. of the differential inclusion
$$
  \begin{cases}
	  -\dot{y}(t) \in N_{C(t)}y(t)+f_1(t,y(t))+\int_{T_0}^tf_2(t,s,y(s))\,ds+ z(t) \;\,
		\text{a.e.}\;t\in I \\
		y(T_0)=x_0.
	\end{cases}
$$
With  Theorem \ref{theor-main} at hands as furnished in the statement of this
Theorem \ref{theor-main}, it is not difficult to see that anyone of the methods in Theorem 4.1 of Tolstonogov \cite{Tolst2017} or in Theorem 4.4 of Castaing and Saidi \cite{Cast-Said} shows the following.  There are sequences $(z_n)_n$ of measurable selections of $\Gamma_{\gamma}$ converging weakly to $z$ in $L^1(I,H)$ and such that, denoting (as above)
$\zeta_{z_n}$ the absolutely continuous solution of $(\mathcal{P}_{f_1,f_2}(x_0,,z_n))$ with initial condition $\zeta_{z_n}(T_0)=x_0$, the inclusion $z_n(t) \in F(t,\zeta_{z_n}(t))$ holds a.e. $t\in I$ for each $n\in \N$ and the sequence $(\zeta_{z_n}(\cdot))_n$ converges uniformly on $I$ to $y(\cdot)$. In particular, any absolutely continuous solution of $(\mathcal{P}_{f_1,f_2}(x_0,\overline{\mathrm{co}}\,F))$ is the uniform limit in $\mathcal{C}(I,H)$ of a sequence of absolutely continuous solutions of
$(\mathcal{P}_{f_1,f_2}(x_0,F))$.

\end{document}